\documentclass[12pt,a4paper]{article}
\usepackage{amsmath}
\usepackage{amssymb}
\usepackage{cases}
\usepackage{bbm}
\usepackage{mathrsfs}
\usepackage{graphicx}
\usepackage{amsfonts}

\usepackage{verbatim}
\usepackage{enumerate}
\usepackage{theorem}
\usepackage{color}
\usepackage{cite}
\makeatletter
\def\tank#1{\protected@xdef\@thanks{\@thanks
 \protect\footnotetext[0]{#1}}}
\def\bigfoot{

 \@footnotetext}
\makeatother

\newcommand{\ea}{\end{array}}

\allowdisplaybreaks

\newtheorem{theorem}{Theorem}[section]
\newtheorem{lem}{Lemma}[section]
\newtheorem{prp}[theorem]{Proposition}
\newtheorem{thm}[theorem]{Theorem}
\newtheorem{cor}[theorem]{Corollary}
\newtheorem{dfn}[theorem]{Definition}

\newtheorem{remark}{Remark}[section]

\newtheorem{Assumption}{Assumption}[section]

\def\beq{\begin{equation}}
\def\nneq{\end{equation}}

\def\bthm{\begin{thm}}
\def\nthm{\end{thm}}

\def\blem{\begin{lem}}
\def\nlem{\end{lem}}
\def\bprf{\begin{proof}}
\def\nprf{\end{proof}}
\def\bprop{\begin{prop}}
\def\nprop{\end{prop}}
\def\brmk{\begin{rem}}
\def\nrmk{\end{rem}}

\def\bexa{\begin{exa}}
\def\nexa{\end{exa}}
\def\bcor{\begin{cor}}
\def\ncor{\end{cor}}

{\theorembodyfont{\rmfamily}
}

\title{Large deviations of conservative stochastic partial differential equations}

\footnotesize{\author{
Ping Chen$^{1,}$\thanks{chenping@mail.ustc.edu.cn},\ \
 Tusheng Zhang$^{2,}$\thanks{Tusheng.Zhang@manchester.ac.uk}\\
 {\em $^1$ School of Mathematical Sciences,}\\
 {\em University of Science and Technology of China,}\\
 {\em Hefei, 230026, China.}\\
 {\em $^2$  Department of Mathematics,}\\
 {\em University of Manchester,}\\
 {\em Oxford Road, Manchester, M13 9PL, UK.}\\
}
\date{}
\newenvironment{proof}{\par\noindent{\bf Proof:}}{\hspace*{\fill}$\blacksquare$\par}
\begin{document}
\maketitle
\noindent \textbf{Abstract:}
In this paper, we establish a large deviation principle for the conservative stochastic partial differential equations, whose solutions are related to stochastic differential equations with interaction. The weak convergence method and the contraction principle in the theory of large deviations  play an important role.


\vspace{4mm}


\vspace{3mm}
\noindent \textbf{Key Words:}
Large deviation principle; conservative stochastic partial differential equations; stochastic differential equations with interaction; weak convergence method.
\numberwithin{equation}{section}
\vskip 0.3cm
\noindent \textbf{AMS Mathematics Subject Classification:} Primary 60H15, 60F10; Secondary 60G57.

\section{Introduction}
Let $T \in (0, \infty)$ be fixed. In this paper, we study the small noise large deviation principle (LDP) of conservative stochastic partial differential equations (SPDEs), which is written as follows:
\begin{eqnarray}\label{I-eq1.2}
\begin{split}
\mathrm{d}\mu_t &= \frac{1}{2}D^2:(A(t, \cdot, \mu_t)\mu_t)\mathrm{d}t - \nabla \cdot (V(t, \cdot, \mu_t)\mu_t)\mathrm{d}t \\
                &\quad - \int_{\Theta} \nabla \cdot (G(t, \cdot, \mu_t, \theta)\mu_t)W(\mathrm{d}\theta, \mathrm{d}t),
\end{split}
\end{eqnarray}
where $\{W_t\}_{t \in [0,T]}$ is a cylindrical Wiener process on some $L^2$ space $L^2(\Theta, \vartheta)$ defined on a complete probability space $(\Omega, \mathcal{F}, \mathbb{P})$. For the precise conditions on $A$, $V$ and $G$, we refer the readers to Section \ref{sec:preliminaries}.

Conservative SPDEs arise as fluctuating continuum models for interacting particle system \cite{GLP}, and have been used to describe the hydrodynamic large deviations of simple exclusion and zero range particle processes, see, for instance, \cite{QRV} and \cite{BKL}. In \cite{FG}, the authors obtained the well-posedness of conservative SPDEs with correlated noise. In \cite{CG}, the authors adopted a duality argument to establish the well-posedness of measure-valued solutions to the nonlinear, nonlocal stochastic Fokker-Planck equations. Additionally, motivated from fluid dynamics in vorticity form, the case of signed measure-valued solutions has been considered in \cite{RV,AX,KPM1} and the references therein. In a recent work \cite{GGK}, the authors showed that under sufficiently nice initial distribution $\mu_0 \in \mathcal{P}_2(\mathbb{R}^{d})$ and appropriate conditions on the coefficients, the each solution to \eqref{I-eq1.2} was given as a superposition of solutions to the following stochastic differential equation (SDE) with interaction:
\begin{align}
\left\{
\begin{aligned}\label{I-eq1.1}
\mathrm{d}X_t(x)&=V(t,X_t(x),\bar{\mu}_t)\mathrm{d}t+\int_{\Theta}G(t,X_t(x),\bar{\mu}_t,\theta)W(\mathrm{d}\theta, \mathrm{d}t),\\
 X_0(x)&=x,\quad \bar{\mu}_t=\mu_{0}\circ(X_t(\cdot))^{-1},\quad x \in \mathbb{R}^{d},\quad t\in[0,T].
\end{aligned}
\right.
\end{align}

SDEs with interaction have been widely studied since the work by Dorogovtsev in \cite{DAA}. In \cite{LMP}, the authors studied the limit behavior of solutions to SDEs with interaction in one-dimensional case. This has been extended to two-demensional case in \cite{MAB}. In \cite{PAY}, the authors proved a Stroock-Varadhan-type support theorem for stochastic flows generated by SDEs with interaction. In \cite{JD}, the authors established the well-posedness of backward stochastic differential equations. We also refer to readers to \cite{JD1,DO} and references therein for further related works.

Large deviations principles (LDP) provide  exponential estimates for the probability of the rare events in terms of some explicit rate function.
It has wide applications in statistics, statistic mechanics, stochastic dynamical systems, financial mathematics etc.  The early framework was introduced by Varadhan in \cite{Varadhan1966Asymptotic}, in which the small noise (also called Freidlin-Wentzell) LDP for finite dimensional diffusion processes were studied. The literature on large deviations is huge, and a list of them can be found, e.g.,  in \cite{WYJT}.

The purpose of this paper is to establish a Freidlin-Wentzell type LDP for the conservative SPDE \eqref{I-eq1.2}. To obtain the LDP of the solutions, we first establish the LDP of the associated SDEs with interaction \eqref{I-eq1.1} and then use the contraction principle in the theory of large deviations. We will adopt the weak convergence method in large deviations introduced in \cite{BDA}. More precisely, we will use the more convenient sufficient conditions given in the paper \cite{MSW}.

The remaining part of this paper is organized as follows. In Section \ref{sec:preliminaries}, we introduce the basic notations and assumptions, recall well-posedness results for the SDE with interaction \eqref{I-eq1.1} and conservative SPDE \eqref{I-eq1.2}. Moreover, we characterize  the trajectory space of the solutions of the SDEs with interaction \eqref{I-eq1.1}. The trajectory space will be the space on which the large deviation is established.  In Section \ref{sec:mainresults}, we present the main results of the paper, and then outline how the results will be proved. Section \ref{LDPcondition2} and Section \ref{LDPcondition1} are devoted to the proof of the LDP for the SDE with interaction \eqref{I-eq1.1}.


\section{Preliminaries}\label{sec:preliminaries}

Throughout this paper, we use the following notations and assumptions.

For $n \in \mathbb{N}$, let $[n]:= \{ 1, \cdots, n  \}$.
For vectors $a$, $b \in \mathbb{R}^{d}$ and $A$, $B \in \mathbb{R}^{d \times d}$, we set $a \cdot b = \sum_{i =1}^{d} a_i b_i$, $|a|= \sqrt{a \cdot a}$,
and $A : B =\sum_{i, j=1}^{d}a_{i,j}b_{i,j}$.
For $a\in \mathbb{R}^{d}$ and $r>0$, denote by $B(a,r)$ the open ball of radius $r$ centred at $a$ in $\mathbb{R}^{d}$.
We use the letter $C_p$ to denote  a generic positive constant depending on some parameter $p$, whose value may change from line to line.

Let $C^{2}(\mathbb{R}^{d}, \mathbb{R})$ be the space of twice continuously differentiable functions from $\mathbb{R}^{d}$ to $\mathbb{R}$.
Let $C_{c}^{2}(\mathbb{R}^{d}, \mathbb{R})$ denote the subspace of $C^{2}(\mathbb{R}^{d}, \mathbb{R})$ of all compactly supported functions.
For $\varphi$, $f_i \in C^{2}(\mathbb{R}^{d}, \mathbb{R})$, $i \in [d]$, define $\nabla \varphi$ and $D^2 \varphi$ as the gradient and Hessian matrix of $\varphi$, respectively, denote by $\nabla \cdot f = \sum_{i=1}^{d}\partial{f_i} / \partial_{x_i}$, where $f = (f_i)_{i \in [d]}$.
Let $C(\mathbb{R}^{d})$ denote the space of continuous functions from $\mathbb{R}^{d}$ to $\mathbb{R}^{d}$.
For some $\delta \in (0, \frac{1}{3})$, let
\[ C_{\delta}(\mathbb{R}^{d}):= \{ f \in C(\mathbb{R}^{d}): \sup_{x \in \mathbb{R}^{d}} \frac{|f(x)|}{1+|x|^{1+\delta}}< \infty   \}.  \]
Let $C([0,T], C_{\delta}(\mathbb{R}^{d}))$ be the space of continuous functions from $[0, T]$ to $C_{\delta}(\mathbb{R}^{d})$, equipped with the maximum norm, defined as follows
\begin{equation}\label{0.1}
 \| f \|_{\infty, T}:= \sup_{t \in [0,T]} \sup_{x \in \mathbb{R}^{d}} \frac{|f(t,x)|}{1+|x|^{1+\delta}},\ \ f \in C([0,T], C_{\delta}(\mathbb{R}^{d})).
 \end{equation}

For $p \geq 1$, let $\mathcal{P}_{p}(\mathbb{R}^{d})$ denote the set of all probability measures on $(\mathbb{R}^{d}, \mathcal{B}(\mathbb{R}^{d}))$ with finite $p$-moment. It is well known that $\mathcal{P}_{p}(\mathbb{R}^{d})$ is a polish space under the Wasserstein $p$-distance
\[ \mathcal{W}_{p}(\mu, \nu)
= \inf_{\pi \in \Pi(\mu, \nu)}(\int_{\mathbb{R}^{d} \times \mathbb{R}^{d}} |x-y|^{p} \pi(\mathrm{d}x, \mathrm{d}y))^{\frac{1}{p}},
\ \mu , \nu \in \mathcal{P}_{p}(\mathbb{R}^{d}),  \]
where $\Pi(\mu, \nu)$ denotes all probability measures on $\mathbb{R}^{d} \times \mathbb{R}^{d}$ with marginals $\mu$ and $\nu$.
Define $\langle \varphi, \mu \rangle$ as the integration of a function $\varphi: \mathbb{R}^{d} \rightarrow \mathbb{R}$ with respect to a measure $\mu \in \mathcal{P}_{p}(\mathbb{R}^{d})$.

Given a measure space $(\Theta, \mathcal{G}, \vartheta)$, denote by $L^2(\Theta, \vartheta):= L^2(\Theta, \mathcal{G}, \vartheta)$ the usual space of all square integrable functions from $\Theta$ to $\mathbb{R}$, equipped with the inner product and norm, which are denoted by $\langle \cdot, \cdot \rangle_{\vartheta}$ and $\| \cdot \|_{\vartheta}$, respectively.
Let $\{W_t\}_{t \in [0, T]}$ be a cylindrical Wiener process on $L^2(\Theta, \vartheta)$ defined on a complete probability space
$(\Omega, \mathcal{F}, \mathbb{F}, \mathbb{P})$ with the augmented filtration $\mathbb{F}:=\{{\mathcal{F}}_t\}_{t \in [0, T]}$ generated by $\{ W_t \}_{t \in [0, T]}$. For an $(\mathcal{F}_t)$-progressively measurable $L^2(\Theta, \vartheta)$-valued process $g(t, \cdot)$, $t \in [0, T]$, with
\[ \int_0^t \| g(s, \cdot) \|_{\vartheta}^2 \mathrm{d}s < \infty  \quad a.s. \]
for every $t \in [0, T]$, the following Ito-type stochastic integral
\[ \int_0^t \int_{\Theta}g(s,\theta)W(\mathrm{d}\theta, \mathrm{d}s):= \int_0^t \Upsilon(s)\mathrm{d}W_s,   \]
is well defined,  where the operator-valued process $ \Upsilon(s): L^2(\Theta, \vartheta)\rightarrow \mathbb{R}$ is defined by
$\Upsilon(s)h = \langle g(s, \cdot), h\rangle_{\vartheta} $ for $h\in L^2(\Theta, \vartheta)$, $s \in [0, T]$.

Now we give the assumptions on the coefficients $V$, $G$ and $A$ appearing in equations \eqref{I-eq1.2} and \eqref{I-eq1.1}.
Let $V: [0,T] \times \mathbb{R}^{d} \times \mathcal{P}_2(\mathbb{R}^{d}) \rightarrow \mathbb{R}^{d}$ ,
$G: [0,T] \times \mathbb{R}^{d} \times \mathcal{P}_2(\mathbb{R}^{d}) \rightarrow (L_2(\Theta, \vartheta))^{d}$
and $A: [0, T] \times \mathbb{R}^{d} \times \mathcal{P}_2(\mathbb{R}^{d}) \rightarrow \mathbb{R}^{d \times d}$
are $\mathcal{B}([0,T]) \otimes \mathcal{B}(\mathbb{R}^{d})  \otimes  \mathcal{B}(\mathcal{P}_2(\mathbb{R}^{d}))$-measurable functions.

\begin{Assumption}\label{Assumption1}
For all $t \in [0, T]$, $x \in \mathbb{R}^{d}$ and $\mu \in \mathcal{P}_2(\mathbb{R}^{d})$
\[  A(t, x, \mu) = ( \langle G_{i}(t, x, \mu, \cdot), G_{j}(t, x, \mu, \cdot) \rangle_{\vartheta})_{i, j\in [d]}.  \]
\end{Assumption}

\begin{Assumption}\label{Assumption2}
The coefficients $V$ and $G$ are Lipschitz continuous with respect to $x$ and $\mu$, that is, there exists $L > 0$ such that for every $t \in [0,T]$, $x$, $y\in \mathbb{R}^{d}$ and $\mu$, $\nu \in \mathcal{P}_2(\mathbb{R}^{d})$
\[ |V(t, x, \mu) - V(t, y, \nu)| + \| |G(t, x, \mu) - G(t, y, \nu) |  \|_{\vartheta} \leq L(|x - y| + \mathcal{W}_2(\mu, \nu)),  \]
and
\[ |V(t, 0, \delta_{0} )| + \| |G(t, 0, \delta_{0})| \|_{\vartheta} \leq L ,  \]
where $\delta_{0}$ denotes the $\delta$-measure at 0 on $\mathbb{R}^{d}$.
\end{Assumption}

\begin{remark}\label{Remark1}
Assumption \ref{Assumption2} implies that for all $t \in [0,T]$, $x \in \mathbb{R}^{d}$
and $\mu \in \mathcal{P}_2(\mathbb{R}^{d})$, the coefficients $V$ and $G$ satisfy
\[ |V(t, x, \mu)| + \| |G(t, x, \mu)| \|_{\vartheta} \leq L(1+ |x| + \mathcal{W}_2(\mu, \delta_{0})),  \]
\end{remark}

Following  \cite[Definition 2.2, 2.5 and 2.6] {GGK}, we introduce the definitions of solutions for equations \eqref{I-eq1.2} and \eqref{I-eq1.1}, and the superposition principle.

\begin{dfn}\label{definition2}
Let $\mu_{0} \in \mathcal{P}_2(\mathbb{R}^{d})$. A continuous $(\mathcal{F}_t)$-adapted process $\{\mu_t\}_{t \in [0, T]}$, in $\mathcal{P}_2(\mathbb{R}^{d})$, is a solution to the conservative SPDE \eqref{I-eq1.2} started from $\mu_0$ if for every
$ \varphi \in C^2_c(\mathbb{R}^{d}, \mathbb{R})$, a.s. the equality
\begin{equation}\label{II-eq2.1}
\begin{split}
\langle \varphi, \mu_{t} \rangle &= \langle \varphi, \mu_{0} \rangle + \frac{1}{2} \int_0^t \langle D^2\varphi: A(s, \cdot, \mu_{s}), \mu_{s}\rangle \mathrm{d}s\\
                                 &+ \int_0^t\langle \nabla\varphi \cdot V(s, \cdot, \mu_{s}), \mu_{s} \rangle \mathrm{d}s + \int_0^t\int_{\Theta}
                                 \langle \nabla\varphi \cdot G(s, \cdot, \mu_{s}, \theta), \mu_{s} \rangle W(\mathrm{d}\theta, \mathrm{d}s)
\end{split}
\end{equation}
holds for every $t \in [0, T]$.

\end{dfn}

\begin{dfn}\label{definition1}
A family of continuous processes $\{ X_t(x) \}_{t \in [0,T]}$, $x \in \mathbb{R}^{d}$, is called a solution to the SDE with interaction \eqref{I-eq1.1} with initial mass distribution $\mu_{0} \in \mathcal{P}_2(\mathbb{R}^{d})$ if the restriction of $X$ to the time interval $[0, t]$ is $\mathcal{B}([0, t]) \otimes \mathcal{B}(\mathbb{R}^{d}) \otimes \mathcal{F}_t$-measurable, $\bar{\mu}_t = \mu_{0}\circ(X_t(\cdot))^{-1} \in \mathcal{P}_2(\mathbb{R}^{d})$ a.s. for all $t \in [0,T]$ and for every $x \in \mathbb{R}^{d}$, a.s.
\[ X_t(x) = x + \int_0^t V(s, X_s(x), \bar{\mu}_s)\mathrm{d}s + \int_0^t \int_{\Theta} G(s, X_s(x), \bar{\mu}_s, \theta)W(\mathrm{d}\theta, \mathrm{d}s)\] for all $t \in [0,T]$.
\end{dfn}

\begin{dfn}\label{definition3}
A continuous process $\{ \mu_{t} \}_{t \in [0, T]}$, in $\mathcal{P}_2(\mathbb{R}^{d})$ started from $\mu_{0}$ is a  superposition solution to the conservative SPDE \eqref{I-eq1.2} or satisfies the superposition principle if there exists a solution $X_{t}(x)$, $t \in [0, T]$, $x \in \mathbb{R}^{d}$, to the SDE with interaction \eqref{I-eq1.1} such that $\mu_{t} = \mu_{0}\circ (X_t(\cdot))^{-1}$, $t \in [0, T]$, a.s.
\end{dfn}

With these definitions, we recall the following results from \cite{GGK}.

\begin{prp}\label{proposition1}
Let Assumption \ref{Assumption2} hold. Then for every $\mu_{0} \in \mathcal{P}_2(\mathbb{R}^{d})$, the SDE with interaction \eqref{I-eq1.1} has a unique solution $X_t(x)$, $t \in [0,T]$, $x \in \mathbb{R}^{d}$. Moreover, there exists a version of $X_{\cdot}(x)$, $x \in \mathbb{R}^{d}$, that is a continuous in $(t, x)$, and for each $p \geq 2$, there exists a constant $C_{L, p, d, T} > 0$ such that for all $x$, $y \in \mathbb{R}^{d}$
\begin{equation}\label{II-eq2.4}
\mathbb{E}[\sup_{t \in [0,T]} | X_t(x)|^p] \leq C_{L, p, d, T}( 1 + \int_{\mathbb{R}^{d}} |x|^{p} \mu_0(\mathrm{d}x) + |x|^{p} ),
\end{equation}
and
\begin{equation}\label{II-eq2.5}
\mathbb{E}[\sup_{t \in [0,T]} | X_t(x) - X_t(y) |^p ] \leq C_{L, p, d, T}|x - y|^p.
\end{equation}
\end{prp}

\begin{prp}\label{proposition6}
Let Assumptions \ref{Assumption1} and \ref{Assumption2} hold. Then for every $\mu_{0} \in \mathcal{P}_2(\mathbb{R}^{d})$, there exists a unique superposition solution $\{\mu_{t}\}_{t \in [0, T]}$ to the conservative SPDE \eqref{I-eq1.2} started from $\mu_{0}$.
\end{prp}

From now on, we will only consider the version $X_t(x)$, $t \in [0, T]$, $x \in \mathbb{R}^{d}$, of a solution to the SDE with interaction \eqref{I-eq1.1} which is continuous in $(t, x)$. Before the end of this section, we present a result on the space of the trajectory  of the solution $X_t(x)$.

\begin{prp}\label{proposition7}
Let Assumption \ref{Assumption2} hold and $X_t(x)$, $t \in [0, T]$, $x \in \mathbb{R}^{d}$ be the unique solution to the SDE with interaction \eqref{I-eq1.1} with initial mass distribution $\mu_{0} \in \mathcal{P}_m(\mathbb{R}^{d})$, $ m > \max \{ \frac{1}{\delta}, 2d \}$ .
Then $X(\omega) \in C([0,T], C_{\delta}(\mathbb{R}^{d}))$ for $\mathbb{P}$-almost all $\omega \in \Omega$.
\end{prp}

We first give a moment estimate, which plays an important role in the proof of Proposition \ref{proposition7}.

\begin{lem}\label{lemma5}
Let Assumption \ref{Assumption2} hold and $X_t(x)$, $t \in [0, T]$, $x \in \mathbb{R}^{d}$ be the unique solution to the SDE with interaction \eqref{I-eq1.1} with initial distribution $\mu_{0} \in \mathcal{P}_m(\mathbb{R}^{d})$, $ m > 2d $.
Then there exists a constant $C_{L, m, d, T}$ such that for every $k \in \mathbb{N}$
\begin{equation}\label{II-eq2.2}
\mathbb{E}[ \sup_{|x| \in [k, k+1)} \sup_{t \in [0,T]} |X_t(x)|^m ] \leq C_{L, m, d, T}(1 + k )^m.
\end{equation}
\end{lem}

\begin{proof}
Let  $\mathcal{R}_n$ denotes a cube in $\mathbb{R}^{d}$ whose length of the edge is $n$. For each $i \in [d]$, $t \in [0,T]$ and $n \in \mathbb{N}$, we have
\begin{eqnarray}\label{II-eq2.3}
& & \int_{\mathcal{R}_1}\int_{\mathcal{R}_1}
          \Big| \frac{ X^{ i}_{t}(nx) - X^{ i}_{t}(ny)}{|x - y| / \sqrt{d}} \Big|^m \mathrm{d}x\mathrm{d}y  \nonumber\\
&=& \int_{\mathcal{R}_n}\int_{\mathcal{R}_n}
          \Big| \frac{ X^{i}_{t}(x^{\prime}) - X^{ i}_{t}(y^{\prime})}{ |x^{\prime} - y^{\prime}| / n\sqrt{d} } \Big|^m
          \cdot n^{-2d}\mathrm{d}x^{\prime}\mathrm{d}y^{\prime}   \nonumber\\
&=& (\sqrt{d})^m n^{m - 2d} \int_{\mathcal{R}_n}\int_{\mathcal{R}_n}
          \frac{|X^{ i}_{t}( x^{\prime}) - X^{ i}_{t}(y^{\prime}) |^m}{|x^{\prime}- y^{\prime} |^m}\mathrm{d}x^{\prime}\mathrm{d}y^{\prime}  \nonumber\\
&= : & B^{ i, n}_t. \nonumber
\end{eqnarray}
Let
\begin{equation*}\label{II-eq2.6}
B^{ n}:= (\sqrt{d})^m n^{m - 2d} \int_{\mathcal{R}_n}\int_{\mathcal{R}_n}
                  \frac{\sup_{t \in [0,T]}|X_{t}( x^{\prime}) - X_{t}(y^{\prime}) |^m}
                       {|x^{\prime}- y^{\prime} |^m}\mathrm{d}x^{\prime}\mathrm{d}y^{\prime}.
\end{equation*}
By Fubini's theorem and \eqref{II-eq2.5}, we have
\begin{eqnarray}\label{II-eq2.7}
\mathbb{E}[ B^{ i, n}_t ]
&\leq & \mathbb{E}[ B^{ n} ]\nonumber\\
&\leq & (\sqrt{d})^m n^{m-2d}C_{L, m, d, T}\int_{\mathcal{R}_n} \int_{\mathcal{R}_n}\mathrm{d}x^{\prime}\mathrm{d}y^{\prime}\\
&\leq & (\sqrt{d})^m n^{m-2d}C_{L, m, d, T} n^{2d}\nonumber\\
&\leq & C_{L, m, d, T} n^m.\nonumber
\end{eqnarray}
Applying the Garsia-Rodemich-Rumsey lemma (see Theorem 1.1 in \cite{WalshJohn}), we obtain that for each $i \in [d]$, $t \in [0,T]$ and $n \in \mathbb{N}$, there exists a set $\Omega_{ i, t, n} \in \mathcal{F}$ such that $\mathbb{P}(\Omega_{ i, t, n})=0 $ and
for all $\omega \in \Omega \backslash \Omega_{ i, t, n}$, $x$, $y \in \mathcal{R}_1$
\begin{eqnarray*}\label{II-eq2.8}
|X^{ i}_t(nx) - X^{ i}_t(ny)|
&\leq & 8\int_0^{|x-y|} \Big( \frac{B^{ i, n}_{t}}{u^{2d}} \Big)^{\frac{1}{m}}\mathrm{d}u\nonumber\\
&  =  & \frac{8}{1- 2d/m} (B^{ i, n}_{t})^{\frac{1}{m}} |x-y|^{1 - \frac{2d}{m}} \nonumber\\
&\leq & \frac{8}{1- 2d/m} (B^{n})^{\frac{1}{m}} |x-y|^{1 - \frac{2d}{m}},
\end{eqnarray*}
which implies that with probability one, for all $x^{\prime}$, $y^{\prime} \in \mathcal{R}_n$
\begin{equation*}\label{II-eq2.9}
|X^{ i}_t(x^{\prime}) - X^{ i}_t(y^{\prime}) |
\leq \frac{8(1/n)^{1 - \frac{2d}{m}}}{1- 2d/m} (B^{n})^{\frac{1}{m}} |x^{\prime}-y^{\prime}|^{1 - \frac{2d}{m}}.
\end{equation*}
Since the mapping: $(x, t) \mapsto X_t(x)$ is continuous, it follows that there exists a set $\Omega_{ i, n} \in \mathcal{F}$ such that
$\mathbb{P}(\Omega_{ i, n})=0 $ and
for all $\omega \in \Omega \backslash \Omega_{ i, n}$, $x^{\prime}$, $y^{\prime} \in \mathcal{R}_n$
\begin{equation}\label{II-eq2.10}
\sup_{t \in [0,T]}|X^{ i}_t(x^{\prime}) - X^{ i}_t(y^{\prime}) |
\leq \frac{8(1/n)^{1 - \frac{2d}{m}}}{1- 2d/m} (B^{n})^{\frac{1}{m}} |x^{\prime}-y^{\prime}|^{1 - \frac{2d}{m}}.
\end{equation}

Now, choose $x_0 \in \mathbb{R}^{d}$ with $|x_0|=k$ and let $n = 2(k+1)$. By \eqref{II-eq2.4}, \eqref{II-eq2.7} and \eqref{II-eq2.10}, we obtain that
\begin{eqnarray*}\label{II-eq2.11}
&  &   \mathbb{E}[ \sup_{|x| \in [k, k+1 )} \sup_{t \in [0,T]} |X_t(x)|^m ]\nonumber\\
&\leq & C_m(\mathbb{E}[\sup_{|x| \in [k, k+1) } \sup_{t \in [0,T]}|X_t(x) - X_t(x_0)|^m]
         + \mathbb{E}[\sup_{t \in [0,T]}|X_t(x_0)|^m])\nonumber\\
&\leq & C_{m, d}(\mathbb{E}[B^{ 2(k+1)}] + \mathbb{E}[\sup_{t \in [0,T]}|X_t(x_0)|^m] )\nonumber\\
&\leq & C_{L, m, d, T}(1 + k )^m.
\end{eqnarray*}

\end{proof}

Now we come back to the proof of Proposition \ref{proposition7}.

\begin{proof}
Since the mapping: $(x, t) \mapsto X_t(x)$ is continuous, it is sufficient to prove that with probability one,
\begin{equation}\label{II-eq2.12}
\lim_{|x| \rightarrow \infty}\sup_{t \in [0, T]} \frac{|X_t(x)|}{1+ |x|^{1+ \delta}} = 0.
\end{equation}
It is obvious that \eqref{II-eq2.12} is equivalent to
\begin{equation}\label{II-eq2.13}
\lim_{n \rightarrow \infty}\sup_{|x| \geq n}\sup_{t \in [0, T]}\frac{|X_t(x)|}{1+ |x|^{1+ \delta}} =0.
\end{equation}
For each $k \in \mathbb{N}$, let
\begin{equation}\label{II-eq2.14}
\eta_{k}:= \sup_{|x| \in [k, k+1)} \sup_{t \in [0, T]}\frac{|X_t(x)|}{1+ |x|^{1+ \delta}}.
\end{equation}
Then
\begin{equation*}\label{II-eq2.15}
\sup_{|x| \geq n}\sup_{t \in [0, T]}\frac{|X_t(x)|}{1+ |x|^{1+ \delta}} = \sup_{k \geq n}\eta_{k},
\end{equation*}
and \eqref{II-eq2.13} is equivalent to $\lim_{k \rightarrow \infty}\eta_{k}=0$.

Now, we verify that with probability one, $\lim_{k \rightarrow \infty}\eta_{k}=0$, completing the proof.
For every $\gamma > 0$ and $k \in \mathbb{N}$, let
\begin{equation*}\label{II-eq2.16}
A_k := \{ \sup_{ |x| \in [k, k+1) } \sup_{t \in [0, T]} |X_t(x)|^m > \gamma ( 1 + k^{ 1+ \delta} )^{m} \}.
\end{equation*}
By Chebyshev's inequality and Lemma \ref{lemma5}, we have
\begin{eqnarray*}\label{II-eq2.17}
\sum_{k=1}^{\infty} \mathbb{P}(A_k)
&\leq & \sum_{k=1}^{\infty} \frac{ \mathbb{E}[ \sup_{ |x| \in [k, k+1) } \sup_{t \in [0, T]} |X_t(x)|^m ] }{\gamma ( 1+ k^{1 + \delta} )^m }   \nonumber\\
&\leq & \sum_{k=1}^{\infty} \frac{ C_{L, m, d, T}( 1 + k )^m }{\gamma( 1+ k^{1 + \delta} )^m}  \nonumber\\
&\leq & C_{L, m, d, T, \gamma}\sum_{k=1}^{\infty}\frac{1}{k^{\delta m}} < \infty, \nonumber
\end{eqnarray*}
here we have used $m > \max \{ \frac{1}{\delta}, 2d \}$ in the last step.\\
Using the Borel-Cantelli lemma, we get that with probability one,
\begin{equation*}\label{II-eq2.18}
     \lim_{k \rightarrow \infty} \eta_{k}^{m}
\leq \lim_{k \rightarrow \infty}\frac{\sup_{|x| \in [k, k+1)}\sup_{t \in [0, T]}|X_t(x)|^m}{(1+ k^{1+\delta})^m} = 0.
\end{equation*}

The proof is complete.
\end{proof}

\section{Statement of main results}\label{sec:mainresults}

In this section, we assume that the initial mass distribution $\mu_0 \in \mathcal{P}_m(\mathbb{R}^{d})$ with $m > \max \{ \frac{1}{\delta}, 2d \}$. For any $\epsilon > 0$, consider the following equation:
\begin{align}
\left\{
\begin{aligned}\label{III-eq3.1}
\mathrm{d}X^{\epsilon}_t(x)&=V(t,X^{\epsilon}_t(x),\bar{\mu}^{\epsilon}_t)\mathrm{d}t
                                +\sqrt{\epsilon}\int_{\Theta}G(t,X^{\epsilon}_t(x),\bar{\mu}^{\epsilon}_t,\theta)W(\mathrm{d}\theta, \mathrm{d}t),\\
X^{\epsilon}_{0}(x)&=x,\quad \bar{\mu}^{\epsilon}_t=\mu_{0}\circ(X^{\epsilon}_t(\cdot))^{-1},\quad x \in \mathbb{R}^{d},\quad t\in[0,T].
\end{aligned}
\right.
\end{align}

By Proposition \ref{proposition1}, the equation \eqref{III-eq3.1} admits a unique solution
$X^{\epsilon}_t(x)$, $t \in [0, T]$, $x \in \mathbb{R}^{d}$.
Moreover, by Proposition \ref{proposition7}, for $\mathbb{P}$-almost all $\omega \in \Omega$,
$X^{\epsilon}(\omega) \in C([0,T], C_{\delta}(\mathbb{R}^{d}))$.
\vskip 0.5cm
For any $h\in L^{2}([0,T], L^2(\Theta, \vartheta))$, consider the so-called skeleton equation:
\begin{align}
\left\{
\begin{aligned}\label{III-eq3.2}
X^{h}_t(x)&= x+ \int_0^t V(s,X^{h}_s(x), \mu^{h}_s)\mathrm{d}s
                                +\int_0^t \int_{\Theta}G(s, X^{h}_s(x), \mu^{h}_s,\theta)h(s, \theta)\vartheta(\mathrm{d}\theta)\mathrm{d}s,\\
X^{h}_{0}(x)&=x,\quad \mu^{h}_s=\mu_{0}\circ(X^{h}_s(\cdot))^{-1},\quad x \in \mathbb{R}^{d},\quad s \in [0,T].
\end{aligned}
\right.
\end{align}

We have the following result:
\begin{prp}\label{proposition2}
Under Assumption \ref{Assumption2}, the skeleton equation \eqref{III-eq3.2} has a unique solution $X^{h}_t(x)$, $t \in [0, T]$, $x \in \mathbb{R}^{d}$. Moreover, $X^{h} \in C([0,T], C_{\delta}(\mathbb{R}^{d}))$.
\end{prp}
\begin{proof}
Let $\tilde{V}(t, x, \mu):= V(t, x, \mu) + \int_{\Theta}G(t, x, \mu, \theta)h(t, \theta)\vartheta(\mathrm{d}\theta)$. For every $t \in [0,T]$, $x$, $y\in \mathbb{R}^{d}$ and $\mu$, $\nu \in \mathcal{P}_2(\mathbb{R}^{d})$, by Assumption \ref{Assumption2},
\begin{eqnarray*}\label{III-eq3.7}
&     &   |\tilde{V}(t, x, \mu) - \tilde{V}(t, y, \nu)| \nonumber\\
&\leq & |V(t, x, \mu) - V(t, y, \nu)| + | \int_{\Theta} (G(t, x, \mu, \theta) - G(t, y, \nu, \theta)) \cdot h(t, \theta)\vartheta(\mathrm{d}\theta)| \nonumber\\
&\leq & L(1 + \|h(t, \cdot) \|_{\vartheta})(|x -y| + \mathcal{W}_{2}(\mu, \nu)). \nonumber
\end{eqnarray*}
Then, using the similar arguments as in the proof of Theorem 2.9 in \cite{GGK} and  in the
proof of Proposition \ref{proposition7}, and using  the fact that
$L(1 + \|h(t, \cdot) \|_{\vartheta}) \in L^2([0, T], \mathbb{R}^{+})$, one can show that there exists a unique solution
$ X^{h} \in C( [0, T], C_{\delta}(\mathbb{R}^{d}) )$ to equation \eqref{III-eq3.2}. We omit the details.
\end{proof}

We now state the main result in this paper.
\begin{theorem}\label{theorem1}
Let Assumption \ref{Assumption2} hold and $X^{\epsilon}$ be the unique solution to equation \eqref{III-eq3.1}. Then the family of $\{X^{\epsilon}\}_{\epsilon > 0}$ satisfies a LDP on the space $C([0,T], C_{\delta}(\mathbb{R}^{d}))$ with rate function
\begin{equation*}\label{III-eq3.3}
 I(g):= \inf_{ \{ h \in L^2([0,T], L^2(\Theta, \vartheta)): g= Y^h  \} } \{ \frac{1}{2} \int_0^T \|h(t,\cdot)\|_{\vartheta}^2 \mathrm{d}t \},
 \ \forall \ g \in C([0,T], C_{\delta}(\mathbb{R}^{d})),
\end{equation*}
with the convention $\inf\{ \emptyset \}= \infty$, here $Y^h$ solves equation \eqref{III-eq3.2}.
\end{theorem}

\begin{proof}
Let $H$ be a Hilbert space such that the imbedding $L^2(\Theta, \vartheta)\subset H$ is Hilbert-Schmdit.  By Proposition \ref{proposition2}, there exists a measurable mapping
$\Gamma^0(\cdot): C([0,T], H) \rightarrow C([0, T], C_{\delta}(\mathbb{R}^{d}))$
such that $ Y^h= \Gamma^0(\int_0^{\cdot}h(s)\mathrm{d}s)$ for $h\in L^2([0,T], L^2(\Theta, \vartheta))$.

Let
\begin{equation*}\label{III-eq3.4}
\mathcal{H}^{N}:= \{h \in L^2([0,T], L^2(\Theta, \vartheta)): \int_0^T \|h(t, \cdot) \|^2_{\vartheta}\mathrm{d}t \leq N \},
\end{equation*}
and
\begin{equation*}\label{III-eq3.5}
\tilde{\mathcal{H}}^{N}:= \{h: h\ \mbox{is}\ L^2(\Theta, \vartheta)\mbox{-valued}\ \mathcal{F}_{t}\mbox{-predictable process such that}\ h(\omega)\in \mathcal{H}^{N},\ \mathbb{P}\text{-}a.s \}.
\end{equation*}
Throughout this paper, $\mathcal{H}^{N}$ is endowed with the weak topology on $L^2([0,T], L^2(\Theta, \vartheta))$ and it is a polish space.

Notice that the $L^2(\Theta, \vartheta)-$cylindrical Wiener process $W_t, t \geq 0$ is now a $H$-valued Wiener process.
By Proposition \ref{proposition1} and Proposition \ref{proposition7}, for every $\epsilon >0$, there exists a measurable mapping
$\Gamma^{\epsilon}: C([0,T], H) \rightarrow C([0,T], C_{\delta}(\mathbb{R}^{d}))$ such that
\[  \Gamma^{\epsilon}( W(\cdot) ) = X^{\epsilon},  \]
and applying the Girsanov theorem, for any $N > 0$ and $h^{\epsilon} \in \tilde{\mathcal{H}}^N$,
\[ Y^{\epsilon}:= \Gamma^{\epsilon}( W(\cdot) + \frac{1}{\sqrt{\epsilon}}\int_0^{\cdot} h^{\epsilon}(s)\mathrm{d}s )  \]
is the solution of the following SDE with interaction:
\begin{align}
\left\{
\begin{aligned}\label{III-eq3.6}
&Y^{\epsilon}_t(x)= x+ \int_0^t V(s, Y^{\epsilon}_s(x), \nu^{\epsilon}_s)\mathrm{d}s
                                +\int_0^t\int_{\Theta}G(s, Y^{\epsilon}_s(x), \nu^{\epsilon}_s,\theta)h^{\epsilon}(s,\theta)\vartheta(\mathrm{d}\theta)\mathrm{d}s\\
& \qquad\quad\quad \ +\sqrt{\epsilon}\int_0^t\int_{\Theta}G(s, Y^{\epsilon}_s(x), \nu_s^{\epsilon})W(\mathrm{d}\theta, \mathrm{d}s),\\
&Y^{\epsilon}_{0}(x)=x,\quad \nu^{\epsilon}_s=\mu_{0}\circ(Y^{\epsilon}_s(\cdot))^{-1},\quad x \in \mathbb{R}^{d},\quad s \in [0,T].
\end{aligned}
\right.
\end{align}

According to Theorem 3.2 in \cite{MSW}, Theorem \ref{theorem1} is established once we have proved:
\begin{itemize}
\item [({\bf LDP1})]\label{LDP1} For every $N < +\infty$, any family $\{h^{\epsilon}\}_{\epsilon >0} \subset \tilde{\mathcal{H}}^{N}$ and any $\gamma > 0$,
    \[ \lim_{\epsilon \rightarrow 0} \mathbb{P}( \| Y^{\epsilon} - Z^{\epsilon} \|_{\infty, T} > \gamma )=0,  \]
    where $ Y^{\epsilon}= \Gamma^{\epsilon}( W(\cdot)+ \frac{1}{\sqrt{\epsilon}}\int_0^{\cdot} h^{\epsilon}(s) \mathrm{d}s ) $
    and $ Z^{\epsilon}= \Gamma^{0}(\int_0^{\cdot} h^{\epsilon}(s)\mathrm{d}s ). $

\item [({\bf LDP2})]\label{LDP2}For every $N < +\infty$ and any family $\{ h^\epsilon \}_{\epsilon >0} \subset \mathcal{H}^{N}$ that converges weakly to some element $h$ in $\mathcal{H}^{N}$ as $\epsilon \rightarrow 0$,
    \[ \lim_{\epsilon \rightarrow 0} \| \Gamma^0(\int_0^{\cdot}h^{\epsilon}(s)\mathrm{d}s) - \Gamma^0(\int_0^{\cdot}h(s)\mathrm{d}s) \|_{\infty, T} = 0.  \]

\end{itemize}

(LDP1) will be verified in Proposition \ref{Proposition5} in Section \ref{LDPcondition2} and (LDP2) will be established
in Proposition \ref{Proposition4} in Section \ref{LDPcondition1}.

\end{proof}

For every $\epsilon > 0$, consider the conservative SPDE:
\begin{align}
\left\{
\begin{aligned}\label{III-eq3.8}
\mathrm{d}\mu^{\epsilon}_t&= \frac{\epsilon}{2}D^2:(A(t, \cdot, \mu^{\epsilon}_t)\mu^{\epsilon}_t)\mathrm{d}t - \nabla \cdot (V(t, \cdot, \mu^{\epsilon}_t)\mu^{\epsilon}_t)\mathrm{d}t \\
& \ \ \ \  - \sqrt{\epsilon}\int_{\Theta} \nabla \cdot (G(t, \cdot, \mu^{\epsilon}_t, \theta)\mu^{\epsilon}_t)W(\mathrm{d}\theta, \mathrm{d}t),\\
\mu^{\epsilon}_0&= \mu_{0}.
\end{aligned}
\right.
\end{align}

By Proposition \ref{proposition6}, there exists a unique superposition solution
$\mu^{\epsilon}_t = \mu_{0}\circ(X^{\epsilon}_t(\cdot))^{-1}$, $t \in [0, T]$, to equation \eqref{III-eq3.8}, where $X^{\epsilon}$ is the unique solution to the corresponding SDE with interaction \eqref{III-eq3.1}. Define the map $F: C([0, T], C_{\delta}(\mathbb{R}^{d})) \rightarrow C([0, T], \mathcal{P}_{2}(\mathbb{R}^{d}))$ as
\[ F(u)(t)= \mu_{0} \circ (u(t, \cdot))^{-1},\ \ u \in C([0,T], C_{\delta}(\mathbb{R}^{d})). \]
It is easy to see that the map is continuous. Then by Theorem \ref{theorem1} and the contraction principle of large deviaions (see Theorem 1.16 in \cite{BD}), we have the following result on the large deviations of conservative SPDEs.
\begin{theorem}\label{corollary1}
Let Assumptions \ref{Assumption1} and \ref{Assumption2} hold and $\mu^{\epsilon}$ be the unique superposition solution to equation \eqref{III-eq3.8}.
Then the family of $\{\mu^{\epsilon}\}_{\epsilon >0}$ satisfies a LDP on the space $C([0, T], \mathcal{P}_{2}(\mathbb{R}^{d}))$ with rate function
\begin{equation*}\label{III-eq3.9}
 J(\phi):= \inf_{ \{ h \in L^2([0,T], L^2(\Theta, \vartheta)): \phi= F(Y^h)  \} } \{ \frac{1}{2} \int_0^T \|h(t,\cdot)\|_{\vartheta}^2 \mathrm{d}t \}, \ \forall \ \phi \in C([0,T], \mathcal{P}_{2}(\mathbb{R}^{d})),
\end{equation*}
with the convention $\inf\{ \emptyset \}= \infty$, here $Y^h$ solves equation \eqref{III-eq3.2}.
\end{theorem}

\section{Verification of (LDP1)}\label{LDPcondition2}

For any $ N < +\infty $ and any family $\{ h^{\epsilon} \}_{ \epsilon > 0} \subset \tilde{\mathcal{H}}^N$, recall that
$ Y^{\epsilon} = \Gamma^{\epsilon}( W(\cdot) + \frac{1}{\sqrt{\epsilon}} \int_0^{\cdot} h^{\epsilon}(s)\mathrm{d}s )$ and
$ Z^{\epsilon} = \Gamma^{0}( \int_0^{\cdot} h^{\epsilon}(s)\mathrm{d}s )$. In this section, we will prove the following result.
\begin{prp}\label{Proposition5}
Let Assumption \ref{Assumption2} hold. Then for any $\gamma > 0$,
\begin{equation*}\label{vII-eq6.1}
\lim_{ \epsilon \rightarrow 0}\mathbb{P}(\| Y^{\epsilon} - Z^{\epsilon}  \|_{\infty, T} > \gamma) = 0.
\end{equation*}
\end{prp}

Before proving Proposition \ref{Proposition5}, we prepare the following Lemmas.

\begin{lem}\label{lemma3}
Let Assumption \ref{Assumption2} hold. Then for any $p \in [2, m]$, there exists a constant $C_{L, p, T, N, d}$ such that for all $x, y \in \mathbb{R}^{d}$
\begin{itemize}
  \item [(i)]\label{itemize1}
  $ \sup_{\epsilon \in (0,1]}\mathbb{E}[ \sup_{t \in [0, T]}|Y^{\epsilon}_t(x)|^p ] \leq C_{L, p, T, N, d}(1 + |x|^p). $
  \item [(ii)]\label{itemize2}
  $ \sup_{\epsilon > 0}\mathbb{E}[ \sup_{t \in [0, T]}|Z^{\epsilon}_t(x)|^p ] \leq C_{L, p, T, N, d}(1 + |x|^p). $
  \item [(iii)]\label{itemize3}
  $ \sup_{\epsilon \in (0, 1]}\mathbb{E}[ \sup_{t \in [0, T]}|Y^{\epsilon}_t(x) - Y^{\epsilon}_t(y)|^p] \leq C_{L, p, T, N, d}|x -y|^p. $
  \item [(iv)]\label{itemize4}
  $ \sup_{\epsilon > 0}\mathbb{E}[ \sup_{t \in [0, T]}|Z^{\epsilon}_t(x) - Z^{\epsilon}_t(y)|^p] \leq C_{L, p, T, N, d}|x -y|^p. $
\end{itemize}
\end{lem}
\vskip 0.4cm
The proof of this lemma is a minor modification of the proof of Theorem 2.9 and Theorem 2.14 in \cite{GGK}. We omit the details.

\begin{lem}\label{lemma4}
Let Assumption \ref{Assumption2} hold. For any $x \in \mathbb{R}^{d}$, suppose that
\begin{equation}\label{vII-eq6.17}
\lim_{\epsilon \rightarrow 0}\mathbb{E}[\sup_{t \in [0, T]}|Y^{\epsilon}_t(x) - Z^{\epsilon}_t(x)|^{m}] = 0.
\end{equation}
Then for any $M \in \mathbb{N}$
\begin{equation*}\label{vII-eq6.18}
\lim_{\epsilon \rightarrow 0}\mathbb{E}[\sup_{t \in [0, T]} \sup_{|x| \leq M }|Y^{\epsilon}_t(x) - Z^{\epsilon}_t(x)|^{m}] = 0.
\end{equation*}
\end{lem}

\begin{proof}
Since $\{x \in \mathbb{R}^{d}: |x| \leq M  \}$ is compact, for every $\kappa \in (0, 1]$, there exist $n_{\kappa} \in \mathbb{N} $ and $\{ x_{i} \}_{i \in [n_{\kappa}]} \subset \{x \in \mathbb{R}^{d}: |x| \leq M  \}$ such that $\{x \in \mathbb{R}^{d}: |x| \leq M  \} \subset \cup _{i=1}^{n_{\kappa}} {B(x_i, \kappa)}$.
Then we have
\begin{eqnarray}\label{vII-eq6.19}
&     &\mathbb{E}[ \sup_{t \in [0, T]} \sup_{|x| \leq M} |Y^{\epsilon}_t(x) - Z^{\epsilon}_t(x)|^m ] \nonumber\\
&  =  &  \mathbb{E}[  \sup_{|x| \leq M}\sup_{t \in [0, T]} |Y^{\epsilon}_t(x) - Z^{\epsilon}_t(x)|^m ] \nonumber\\
&\leq &  \mathbb{E}[ \sup_{x \in \cup_{i=1}^{n_{\kappa}} B(x_i, \kappa)} \sup_{t \in [0, T]}|Y^{\epsilon}_t(x) - Z^{\epsilon}_t(x)|^m ] \nonumber\\
&\leq &  \mathbb{E}[ \vee_{i=1}^{n_{\kappa}}( \sup_{x \in B(x_i, \kappa )} \sup_{t \in [0, T]}|Y^{\epsilon}_t(x) - Z^{\epsilon}_t(x)|^m ) ] \nonumber\\
&\leq & 3^{m-1}\Big\{ \mathbb{E}[ \vee_{i=1}^{n_{\kappa}}( \sup_{x \in B(x_i, \kappa )} \sup_{t \in [0, T]}|Y^{\epsilon}_t(x) - Y^{\epsilon}_t(x_i)|^m )] \nonumber\\
&     & +\mathbb{E}[ \vee_{i=1}^{n_{\kappa}}(\sup_{t \in [0, T]}|Y^{\epsilon}_t(x_i) - Z^{\epsilon}_t(x_i)|^m )  ] \nonumber\\
&     & +\mathbb{E}[ \vee_{i=1}^{n_{\kappa}}( \sup_{x \in B(x_i, \kappa )} \sup_{t \in [0, T]}|Z^{\epsilon}_t(x_i) - Z^{\epsilon}_t(x)|^m )  ]\Big\} \nonumber\\
& =:  & 3^{m-1}\Big\{ I_1 + I_2 + I_3 \Big\},
\end{eqnarray}
where for each $a$, $b \in \mathbb{R}$, $a \vee b := \max\{a, b \} $.

 By $(iii)$ in Lemma \ref{lemma3}, using an argument similar to the proof of  \eqref{II-eq2.10}, we obtain that with probability one, for all $x$, $y \in \mathcal{R}_{2( M+ \kappa)}$
\begin{equation*}\label{vII-eq6.20}
     \sup_{t \in [0, T]}|Y^{\epsilon}_t(x) - Y^{\epsilon}_t(y)|
\leq \frac{8d(1/2( M+ \kappa))^{1-\frac{2d}{m}}}{1-2d/m}(B^{\epsilon, 2( M+ \kappa)})^{\frac{1}{m}}|x -y|^{1-\frac{2d}{m}},
\end{equation*}
where
\begin{equation*}\label{vII-eq6.21}
B^{\epsilon, 2( M+ \kappa)} = (\sqrt{d})^m 2^{m -2d}( M+ \kappa)^{m -2d} \int_{\mathcal{R}_{2( M+ \kappa)}}\int_{\mathcal{R}_{2( M+ \kappa)}}
                                               \frac{\sup_{t \in [0, T]}|Y^{\epsilon}_t(x) - Y^{\epsilon}_t(y)|^m}{|x -y|m}\mathrm{d}x\mathrm{d}y,
\end{equation*}
and $\sup_{\epsilon \in (0, 1]}\mathbb{E}[B^{\epsilon, 2( M+ \kappa)}] \leq C_{L, m, T, N, d} (M + \kappa)^m$.
It follows that for any $\epsilon \in (0, 1]$,
\begin{eqnarray}\label{vII-eq6.23}
         I_{1}
&\leq & \frac{8^md^m(1/2( M+ \kappa))^{m-2d}}{(1-2d/m)^m} \cdot \kappa^{m-2d} \cdot \mathbb{E}[B^{\epsilon, 2( M+ \kappa)}]\nonumber\\
&\leq & C_{L, m, T, N, d, M}\kappa^{m-2d}.
\end{eqnarray}
Similarly, using $(iv)$ in Lemma \ref{lemma3},  for any $\epsilon \in (0, 1]$, we have
\begin{equation}\label{vII-eq6.24}
I_{3} \leq  C_{L, m, T, N, d, M}\kappa^{m-2d}
\end{equation}
Obviously,
\begin{equation}\label{vII-eq6.25}
I_{2} \leq \sum_{i =1}^{n_{\kappa}}\mathbb{E}[\sup_{t \in [0, T]} | Y^{\epsilon}_t(x_i) - Z^{\epsilon}_t(x_i) |^m].
\end{equation}
Substituting \eqref{vII-eq6.23}, \eqref{vII-eq6.24} and \eqref{vII-eq6.25} into \eqref{vII-eq6.19}, we get that for any $\epsilon \in (0, 1]$
\begin{eqnarray}\label{vII-eq6.26}
&    &  \mathbb{E}[ \sup_{t \in [0, T]} \sup_{|x| \leq M} |Y^{\epsilon}_t(x) - Z^{\epsilon}_t(x)|^m ] \\
&\leq&  C_{L, m, T, N, d, M}\kappa^{m-2d} +
     3^{m-1}\sum_{i =1}^{n_{\kappa}}\mathbb{E}[\sup_{t \in [0, T]} | Y^{\epsilon}_t(x_i) - Z^{\epsilon}_t(x_i)|^m]. \nonumber\
\end{eqnarray}

Therefore, for any $\alpha > 0$, choose a sufficiently small $\kappa_0$ such that $C_{L, m, T, N, d, M}{\kappa_0}^{m-2d} < \frac{\alpha}{2}$ and then there exists $\beta \in (0, 1)$ such that for any $\epsilon \in (0, \beta)$,
\begin{equation*}\label{vII-eq6.27}
 \sum_{i =1}^{n_{\kappa_{0}}}\mathbb{E}[\sup_{t \in [0, T]} | Y^{\epsilon}_t(x_i) - Z^{\epsilon}_t(x_i) |^m] < \frac{\alpha}{2 \cdot 3^{m-1}}.
\end{equation*}
It follows from (\ref{vII-eq6.26}) that  for any $\epsilon \in (0, \beta)$,
\begin{eqnarray*}\label{vII-eq6.28}
\mathbb{E}[ \sup_{t \in [0, T]} \sup_{|x| \leq M} |Y^{\epsilon}_t(x) - Z^{\epsilon}_t(x)|^m ] < \alpha.
\end{eqnarray*}
The proof is complete.

\end{proof}
\vskip 0.4cm
\noindent {\bf Proof of Proposition \ref{Proposition5}}.
\begin{proof}
By Chebyshev's inequality, it suffices to show that
\begin{equation}\label{vII-eq6.29}
\lim_{\epsilon \rightarrow 0} \mathbb{E}[ \| Y^{\epsilon} - Z^{\epsilon}  \|_{\infty, T}^{m}] = 0.
\end{equation}
For every $M \in \mathbb{N}$, we have
\begin{eqnarray}\label{vII-eq6.2}
&   & \mathbb{E}[ \| Y^{\epsilon} - Z^{\epsilon}  \|_{\infty, T}^{m} ] \\
& = & \mathbb{E}[\sup_{t \in [0, T]}\sup_{x \in \mathbb{R}^{d}} \frac{|Y^{\epsilon}_t(x) - Z^{\epsilon}_t(x)|^m}{(1 + |x|^{1+\delta})^m}] \nonumber\\
& = & \mathbb{E}[\sup_{t \in [0, T]}\Big(\sup_{|x| < M} \frac{|Y^{\epsilon}_t(x) - Z^{\epsilon}_t(x)|^m}{(1 + |x|^{1+\delta})^m}  \vee
                              \sup_{|x|\geq M} \frac{|Y^{\epsilon}_t(x) - Z^{\epsilon}_t(x)|^m}{(1 + |x|^{1+\delta})^m}  \Big)] \nonumber\\
&\leq& \mathbb{E}[\sup_{t \in [0, T]}\sup_{|x| < M} \frac{|Y^{\epsilon}_t(x) - Z^{\epsilon}_t(x)|^m}{(1 + |x|^{1 + \delta})^m}  \vee
       \sup_{t \in [0, T]}\sup_{|x|\geq M} \frac{|Y^{\epsilon}_t(x) - Z^{\epsilon}_t(x)|^m}{(1 + |x|^{1 + \delta})^m}] \nonumber\\
&\leq& \mathbb{E}[\sup_{t \in [0, T]}\sup_{|x| < M} |Y^{\epsilon}_t(x) - Z^{\epsilon}_t(x)|^m]  +
       \mathbb{E}[\sup_{t \in [0, T]}\sup_{|x|\geq M} \frac{|Y^{\epsilon}_t(x) - Z^{\epsilon}_t(x)|^m}{(1 + |x|^{1 + \delta})^m}]. \nonumber\
\end{eqnarray}

We first prove that for every $M \in \mathbb{N}$
\begin{equation}\label{vII-eq6.11}
\lim_{\epsilon \rightarrow 0}\mathbb{E}[\sup_{t \in [0, T]}\sup_{|x| < M}|Y^{\epsilon}_t(x)-Z^{\epsilon}_t(x)|^m ]=0 .
\end{equation}
By Lemma \ref{lemma4}, it suffices to show that for each $x \in \mathbb{R}^{d}$,
\begin{equation}\label{vII-eq6.5}
\lim_{\epsilon \rightarrow 0} \mathbb{E}[\sup_{t\in [0, T]} |Y^{\epsilon}_t(x) - Z^{\epsilon}_t(x)|^m ] = 0.
\end{equation}
Using H\"older inequality, the Burkholder--Davis--Gundy inequality, Assumption \ref{Assumption2}, Remark \ref{Remark1} and the fact that $\{h^{\epsilon}\}_{\epsilon >0} \subset \mathcal{\tilde{H}}^N$, we have for every $\epsilon \in (0, 1]$ and $x \in \mathbb{R}^d$
\begin{eqnarray*}\label{vII-eq6.7}
&    &   \mathbb{E}[ \sup_{t \in [0, T]}| Y^{\epsilon}_t(x) - Z^{\epsilon}_t(x) |^m ] \nonumber\\
&\leq& 3^{m-1}\Big\{\mathbb{E}[(\int_0^T |V(s, Y^{\epsilon}_s(x), \nu^{\epsilon}_s) - V(s, Z^{\epsilon}_s(x), \mu^{\epsilon}_s)|\mathrm{d}s)^m]\nonumber\\ &    & + \mathbb{E}[(\int_0^T |\int_{\Theta} (G( s, Y^{\epsilon}_s(x), \nu^{\epsilon}_s, \theta )- G( s, Z^{\epsilon}_s(x), \mu^{\epsilon}_s, \theta)) \cdot h^{\epsilon}(s, \theta)\vartheta(\mathrm{d}\theta)| \mathrm{d}s)^m] \nonumber\\
&    & + \epsilon^{\frac{m}{2}}\mathbb{E}[\sup_{t \in [0, T]}
          |\int_0^t\int_{\Theta}G(s, Y^{\epsilon}_s(x), \nu^{\epsilon}_s,\theta)W(\mathrm{d}\theta\mathrm{d}s)|^m]\Big\}\nonumber\\
&\leq& C_{m, T}\mathbb{E}[ \int_0^T |V(s, Y^{\epsilon}_s(x), \nu^{\epsilon}_s) - V(s, Z^{\epsilon}_s(x), \mu^{\epsilon}_s)|^m \mathrm{d}s]\nonumber\\
&    & +C_{m}\mathbb{E}[ (\int_0^T \| | G( s, Y^{\epsilon}_s(x), \nu^{\epsilon}_s, \cdot )
                                   - G( s, Z^{\epsilon}_s(x), \mu^{\epsilon}_s, \cdot )| \|_{\vartheta}^2  \mathrm{d}s)^{\frac{m}{2}}
                              \cdot (\int_0^T \| h^{\epsilon}(s, \cdot) \|^2_{\vartheta}\mathrm{d}s)^{\frac{m}{2}} ] \nonumber\\
&    & +\epsilon^{\frac{m}{2}} C_{m, d}
       \mathbb{E}[(\int_0^T \| |G(s, Y^{\epsilon}_s(x), \nu^{\epsilon}_s, \cdot)| \|_{\vartheta}^2 \mathrm{d}s)^{\frac{m}{2}}] \nonumber\\
&\leq& C_{L, m, T, N}\mathbb{E}[\int_0^T |Y^{\epsilon}_s(x) - Z^{\epsilon}_s(x)|^m + \mathcal{W}_2^m(\nu^{\epsilon}_s, \mu^{\epsilon}_s)\mathrm{d}s ] \nonumber\\
&    & + \epsilon^{\frac{m}{2}}C_{L, m, T, d}\mathbb{E}[\int_0^T 1+|Y^{\epsilon}_s(x)|^m + \mathcal{W}_2^m(\nu^{\epsilon}_s, \delta_{0})\mathrm{d}s ]
\nonumber\\
&\leq&  C_{L, m, T, N}\int_0^T \mathbb{E}[\sup_{r \in [0,s]}|Y^{\epsilon}_r(x) - Z^{\epsilon}_r(x)|^m]\mathrm{d}s
                                          + C_{L, m, T, N}\int_0^T\mathbb{E}[\mathcal{W}_2^m(\nu^{\epsilon}_s, \mu^{\epsilon}_s)]\mathrm{d}s\nonumber\\
&    &  + \epsilon^{\frac{m}{2}}C_{L, m, T, d}\int_0^T \mathbb{E}[ 1 + |Y^{\epsilon}_s(x)|^m
                                          + (\int_{\mathbb{R}^{d}}|Y^{\epsilon}_s(y)|^2\mu_{0}(\mathrm{d}y))^{\frac{m}{2}}]\mathrm{d}s \nonumber\\
&\leq&   C_{L, m, T, N}\int_0^T \mathbb{E}[\sup_{r \in [0,s]}|Y^{\epsilon}_r(x) - Z^{\epsilon}_r(x)|^m]\mathrm{d}s
                                          + C_{L, m, T, N}\int_0^T\mathbb{E}[\mathcal{W}_2^m(\nu^{\epsilon}_s, \mu^{\epsilon}_s)]\mathrm{d}s\nonumber\\
&    & + \epsilon^{\frac{m}{2}}C_{L, m, T, N, d}( 1+ |x|^m), \nonumber\
\end{eqnarray*}
here, for the last step we have used $(i)$ in Lemma \ref{lemma3} and the fact that $\int_{\mathbb{R}^{d}}|y|^m\mu_{0}(\mathrm{d}y)<\infty $ .\\
By Gronwall's lemma, we get for each $\epsilon \in (0, 1]$ and $x \in \mathbb{R}^{d}$
\begin{equation}\label{vII-eq6.8}
\mathbb{E}[ \sup_{t \in [0, T]} |Y^{\epsilon}_t(x) - Z^{\epsilon}_t(x) |^m ]
\leq C_{L, m, T, N, d}( \int_0^T \mathbb{E}[\mathcal{W}_2^m(\nu^{\epsilon}_s, \mu^{\epsilon}_s)]\mathrm{d}s + \epsilon^{\frac{m}{2}}(1 + |x|^m) ).
\end{equation}
In order to bound the integral on the right side, using (\ref{vII-eq6.8}) we estimate
\begin{eqnarray*}\label{vII-eq6.9}
      \mathbb{E}[ \sup_{s \in [0, T]}\mathcal{W}_2^m(\nu^{\epsilon}_s, \mu^{\epsilon}_s)]
&\leq& \mathbb{E}[\sup_{s \in [0, T]} (\int_{\mathbb{R}^{d}} | Y^{\epsilon}_s(x) - Z^{\epsilon}_s(x) |^2 \mu_{0}(\mathrm{d}x))^{\frac{m}{2}}] \nonumber\\
&\leq& \int_{\mathbb{R}^{d}} \mathbb{E}[ \sup_{s \in [0, T]}| Y^{\epsilon}_s(x) - Z^{\epsilon}_s(x) |^m ] \mu_{0}(\mathrm{d}x) \nonumber\\
&\leq& C_{L, m, T, N, d}\int_{\mathbb{R}^{d}}( \int_0^T \mathbb{E}[\mathcal{W}_2^m(\nu^{\epsilon}_s, \mu^{\epsilon}_s)]\mathrm{d}s
                                               + \epsilon^{\frac{m}{2}}(1 + |x|^m) )\mu_{0}(\mathrm{d}x) \nonumber\\
&\leq& C_{L, m, T, N, d}\int_0^T \mathbb{E}[ \sup_{r \in [0, s]}\mathcal{W}_2^m(\nu^{\epsilon}_r, \mu^{\epsilon}_r)]\mathrm{d}s
       +  C_{L, m, T, N, d}\epsilon^{\frac{m}{2}}.  \nonumber
\end{eqnarray*}
Thus, Gronwall's lemma yields
\begin{equation*}\label{vII-eq6.10}
\mathbb{E}[ \sup_{s \in [0, T]}\mathcal{W}_2^m(\nu^{\epsilon}_s, \mu^{\epsilon}_s)] \leq C_{L, m, T, N, d}\epsilon^{\frac{m}{2}}.
\end{equation*}
Combining with the inequality above with \eqref{vII-eq6.8} and letting $\epsilon$ tend $0$, we get \eqref{vII-eq6.5}.

Next, we prove that
\begin{equation}\label{vII-eq6.13}
\lim_{M \rightarrow \infty}\sup_{\epsilon \in (0,1]}
                           \mathbb{E}[\sup_{t \in [0, T]}\sup_{|x|\geq M}\frac{|Y^{\epsilon}_t(x) - Z^{\epsilon}_t(x)|^m}{(1 + |x|^{1 + \delta})^{m}}]=0.
\end{equation}
By Lemma \ref{lemma3}, using an argument similar to that proving Lemma \ref{lemma5} shows that there exists a constant
$C_{L, m, T, N, d}$ such that for every $k \in \mathbb{N}$
\begin{equation*}\label{vII-eq6.14}
     \sup_{\epsilon \in (0, 1]}\mathbb{E}[\sup_{|x| \in [k, k+1)} \sup_{t\in [0, T]}| Y^{\epsilon}_t(x) - Z^{\epsilon}_t(x)|^m]
\leq C_{L, m, T, N, d}(1+ k)^m.
\end{equation*}
It follows that
\begin{eqnarray}\label{vII-eq6.15}
&   &\sup_{\epsilon \in (0, 1]}\mathbb{E}[\sup_{t \in [0,T]}\sup_{|x| \geq M}
                                          \frac{|Y^{\epsilon}_t(x) - Z^{\epsilon}_t(x)|^m}{(1 + |x|^{1 + \delta})^{m}}]   \nonumber\\
& = &\sup_{\epsilon \in (0, 1]}\mathbb{E}[\sup_{|x| \geq M}\sup_{t \in [0,T]}
                                          \frac{|Y^{\epsilon}_t(x) - Z^{\epsilon}_t(x)|^m}{(1 + |x|^{1 + \delta})^{m}} ]  \nonumber\\
& = &\sup_{\epsilon \in (0,1]}\mathbb{E}[\sup_{k \geq M } \sup_{|x| \in [k, k+1)} \sup_{t\in [0, T]}
                                          \frac{|Y^{\epsilon}_t(x) - Z^{\epsilon}_t(x)|^m}{(1 + |x|^{1 + \delta})^{m}} ]   \nonumber\\
&\leq& \sum_{k=M}^{\infty}\sup_{\epsilon \in (0,1]}\mathbb{E}[ \sup_{|x| \in [k, k+1)} \sup_{t\in [0, T]}
                                                              \frac{|Y^{\epsilon}_t(x) - Z^{\epsilon}_t(x)|^m}{(1 + |x|^{1 + \delta})^{m}}] \nonumber\\
&\leq& \sum_{k=M}^{\infty}\frac{ \sup_{\epsilon \in (0,1]}\mathbb{E}[\sup_{|x| \in [k, k+1)}\sup_{t\in [0, T]}
                                                               |Y^{\epsilon}_t(x) - Z^{\epsilon}_t(x)|^m]} {( 1 + k^{1+\delta} )^m} \nonumber\\
&\leq& \sum_{k=M}^{\infty}\frac{C_{L, m, T, N, d}(1+ k)^m}{( 1 + k^{1+\delta} )^m}
\leq C_{L, m, T, N, d}\sum_{k=M}^{\infty}\frac{1}{k^{{\delta}m }}.
\end{eqnarray}
Since ${\delta}m > 1$, letting $M \rightarrow \infty$ in \eqref{vII-eq6.15}, we get \eqref{vII-eq6.13}.

Now, we are in the position to complete the proof of the Proposition.
By \eqref{vII-eq6.11} and \eqref{vII-eq6.13}, letting $\epsilon \rightarrow 0 $, and then letting $M \rightarrow\infty$ in \eqref{vII-eq6.2}, we obtain
\begin{equation*}\label{vII-eq6.16}
\lim_{\epsilon \rightarrow 0}\mathbb{E}[ \| Y^{\epsilon} - Z^{\epsilon}  \|_{\infty, T}^{m} ] =0.
\end{equation*}

The proof is complete.

\end{proof}

\section{Verification of (LDP2)}\label{LDPcondition1}

In this section, we verify (LDP2), namely, we will prove the following result.
\begin{prp}\label{Proposition4}
Let Assumption \ref{Assumption2} hold. For every $N < +\infty$ and any family $\{ h^\epsilon \}_{\epsilon >0} \subset \mathcal{H}^{N}$ that converges weakly to some element $h$ in $\mathcal{H}^{N}$ as $\epsilon \rightarrow 0$,  $X^{h^\epsilon}$ converges to $X^{h}$
in the space $C([0, T], C_{\delta}(\mathbb{R}^{d}))$, where $X^h$ solve equation \eqref{III-eq3.2}, and $X^{h^\epsilon}$ solve equation \eqref{III-eq3.2} with $h$ replaced by $h^\epsilon$.
\end{prp}

\begin{proof}
For every $\epsilon > 0$ and $M \in \mathbb{N}$, using an argument similar to that proving \eqref{vII-eq6.2} shows that
\begin{eqnarray}\label{vI-eq5.1}
&   & \sup_{t \in [0, T]}\sup_{x \in \mathbb{R}^{d}} \frac{|X^{h^\epsilon}_t(x) - X^h_t(x)|}{1 + |x|^{1+\delta}} \nonumber\\
&\leq& \sup_{t \in [0, T]}\sup_{|x| \geq M} \frac{|X^{h^\epsilon}_t(x) - X^h_t(x)|}{1 + |x|^{1 + \delta}}  \vee
       \sup_{t \in [0, T]}\sup_{|x|< M} \frac{|X^{h^\epsilon}_t(x) - X^h_t(x)|}{1 + |x|^{1 + \delta}}\nonumber\\
&\leq& \sup_{\epsilon > 0}\sup_{t \in [0, T]}\sup_{|x| \geq M} \frac{|X^{h^\epsilon}_t(x) - X^h_t(x)|}{1 + |x|^{1 + \delta}} \vee
        \sup_{t \in [0, T]}\sup_{|x|< M}|X^{h^\epsilon}_t(x) - X^h_t(x)|\nonumber\\
& =: & I_1^{M} \vee I_2^{M, \epsilon}.
\end{eqnarray}

We first verify that
\begin{equation}\label{vI-eq5.16}
\lim_{M \rightarrow \infty}I_1^{M} = 0.
\end{equation}
By $(ii)$ in Lemma \ref{lemma3}, we have
\begin{eqnarray}\label{vI-eq5.17}
       I_1^{M}
&\leq &  \sup_{|x| \geq M}\frac{C_{L, T, N}(1+ |x|)}{1 + |x|^{1 + \delta }} \nonumber\\
&\leq & \frac{C_{L, T, N}}{ M^{ \delta }} .
\end{eqnarray}
So letting $M \rightarrow \infty$, we obtain \eqref{vI-eq5.16}.

Next, we verify that for every $M \in \mathbb{N}$,
\begin{equation}\label{vI-eq5.18}
\lim_{\epsilon \rightarrow 0}  I_2^{M, \epsilon} = 0.
\end{equation}
By $(iv)$ in Lemma \ref{lemma3} and using a similar argument as in the proof of Lemma \ref{lemma4}, we only need to show that
for every $x \in \mathbb{R}^{d}$,
\begin{equation}\label{vI-eq5.19}
\lim_{\epsilon \rightarrow 0}\sup_{t \in [0,T]}|X^{h^\epsilon}_t(x) - X^{h}_t(x)| = 0.
\end{equation}
For every $\epsilon >0$ and $x \in \mathbb{R}^{d}$, by Assumption \ref{Assumption2}, H\"older inequality and the fact that
$\{h^{\epsilon}\}_{\epsilon > 0} \subset \mathcal{H}^N$, we have
\begin{eqnarray}\label{vI-eq5.8}
&    &     \sup_{t \in [0, T]}|X^{h^\epsilon}_t(x) - X^{h}_t(x)|^2 \nonumber\\
&\leq& 3 \Big\{ ( \int_0^T |V(s, X^{h^\epsilon}_s(x), \mu^{h^\epsilon}_s) - V(s, X^{h}_s(x), \mu^{h}_s)|\mathrm{d}s )^2 \nonumber\\
&    & + ( \int_0^T |\int_{\Theta} (G(s, X^{h^\epsilon}_s(x), \mu^{h^\epsilon}_s, \theta) - G(s, X^{h}_s(x), \mu^{h}_s, \theta))
         \cdot h^{\epsilon}(\theta, s) \vartheta(\mathrm{d}\theta)| \mathrm{d}s )^2  \nonumber\\
&    & +\sup_{t \in [0, T]} | \int_0^t \int_{\Theta}G(s, X^{h}_s(x), \mu^{h}_s, \theta)
         \cdot (h^\epsilon(\theta, s) - h(\theta, s)) \vartheta(\mathrm{d}\theta)  \mathrm{d}s  |^2 \Big\} \nonumber\\
&\leq& 3T \int_0^T |V(s, X^{h^\epsilon}_s(x), \mu^{h^\epsilon}_s) - V(s, X^{h}_s(x), \mu^{h}_s)|^2 \mathrm{d}s \\
&    & +3(\int_0^T \| |G(s, X^{h^\epsilon}_s(x), \mu^{h^\epsilon}_s, \cdot) - G(s, X^{h}_s(x), \mu^{h}_s, \cdot) | \|_{\vartheta}^2 \mathrm{d}s )
          \cdot ( \int_0^T \| h^\epsilon(\cdot, s) \|_{\vartheta}^2 \mathrm{d}s ) \nonumber\\
&    & +3\sup_{t \in [0, T]} | \int_0^t \int_{\Theta}G(s, X^{h}_s(x), \mu^{h}_s, \theta)
         \cdot (h^\epsilon(\theta, s) - h(\theta, s)) \vartheta(\mathrm{d}\theta)  \mathrm{d}s  |^2.  \nonumber\\
&\leq& C_{L, T, N} \int_0^T \sup_{r \in [0, s]}| X^{h^\epsilon}_r(x) - X^{h}_r(x) |^2 \mathrm{d}s
       + C_{L, T, N} \int_0^T \mathcal{W}_2^2( \mu^{h^\epsilon}_s, \mu^{h}_s )\mathrm{d}s \nonumber\\
&    & + 3\sup_{t \in [0, T]} | \int_0^t \int_{\Theta}G(s, X^{h}_s(x), \mu^{h}_s, \theta)
         \cdot (h^{\epsilon}(\theta, s) - h(\theta, s)) \vartheta(\mathrm{d}\theta)  \mathrm{d}s  |^2.  \nonumber
\end{eqnarray}

For each $\epsilon > 0$, $t \in [0, T]$ and $x \in \mathbb{R}^{d}$, define
\[F_{\epsilon}(t, x):= | \int_0^t \int_{\Theta}G(s, X^{h}_s(x), \mu^{h}_s, \theta)
         \cdot (h^{\epsilon}(\theta, s) - h(\theta, s)) \vartheta(\mathrm{d}\theta) \mathrm{d}s  |.  \]
Then applying Gronwall's lemma to \eqref{vI-eq5.8} yields
\begin{equation}\label{vI-eq5.9}
\sup_{t \in [0, T]}| X^{h^\epsilon}_t(x) - X^{h}_t(x) |^2 \leq C_{L, T, N}(\int_0^T \mathcal{W}_2^2( \mu^{h^{\epsilon}}_s, \mu^{h}_s )\mathrm{d}s
                                                        + \sup_{t \in [0, T]}F_{\epsilon}^{2}(t, x)).
\end{equation}
Using the definition of the Wasserstein distance, we further estimate
\begin{eqnarray*}\label{vI-eq5.10}
      \sup_{t \in [0, T]} \mathcal{W}_2^2( \mu^{h^{\epsilon}}_t, \mu^{h}_t )
&\leq& \sup_{t \in [0, T]} \int_{\mathbb{R}^{d}} |X^{h^{\epsilon}}_t(x) - X^{h}_t(x)|^2 \mu_{0}(\mathrm{d}x) \nonumber\\
&\leq& \int_{\mathbb{R}^{d}} \sup_{t \in [0, T]} |X^{h^{\epsilon}}_t(x) - X^{h}_t(x)|^2 \mu_{0}(\mathrm{d}x)  \nonumber\\
&\leq& C_{L, T, N}\int_{\mathbb{R}^{d}} ( \int_0^T \mathcal{W}_2^2( \mu^{h^{\epsilon}}_s, \mu^{h}_s )\mathrm{d}s
                                                        + \sup_{t \in [0, T]}F_{\epsilon}^{2}(t, x) )\mu_{0}(\mathrm{d}x) \nonumber\\
&\leq& C_{L, T, N}\int_0^T \sup_{r \in [0, s]}\mathcal{W}_2^2( \mu^{h^{\epsilon}}_r, \mu^{h}_r )\mathrm{d}s
       + C_{L, T, N}\int_{\mathbb{R}^{d}}\sup_{t \in [0, T]}F_{\epsilon}^{2}(t, x) \mu_{0}(\mathrm{d}x). \nonumber
\end{eqnarray*}
Hence, by Gronwall's lemma
\begin{equation*}\label{vI-eq5.11}
\sup_{t \in [0, T]} \mathcal{W}_2^2( \mu^{h^\epsilon}_t, \mu^{h}_t )
\leq  C_{L, T, N}\int_{\mathbb{R}^{d}}\sup_{t \in [0, T]}F_{\epsilon}^{2}(t, x) \mu_{0}(\mathrm{d}x).
\end{equation*}
Combing the inequality above with \eqref{vI-eq5.9}, we obtain
\begin{equation}\label{vI-eq5.12}
\sup_{t \in [0, T]}| X^{h^\epsilon}_t(x) - X^{h}_t(x) |^2 \leq C_{L, T, N}(\int_{\mathbb{R}^{d}}\sup_{t \in [0, T]}F_{\epsilon}^{2}(t, x) \mu_{0}(\mathrm{d}x)
                                                        + \sup_{t \in [0, T]}F_{\epsilon}^{2}(t, x)).
\end{equation}

Now, we prove that the right-hand side of the above inequality tends to $0$ as $\epsilon$ tending to $0$.
By Remark \ref{Remark1}, the definition of the Wasserstein distance and $(ii)$ in Lemma \ref{lemma3}, we get that for every $x \in \mathbb{R}^{d}$,
\begin{eqnarray*}\label{vI-eq5.13}
        \sup_{t \in [0, T]} \| |G(t, X^{h}_t(x), \mu^{h}_t, \cdot) | \|_{\vartheta}^{2}
&\leq & C_{L}( 1 + \sup_{t \in [0 ,T]}|X^{h}_t(x)|^2 + \sup_{t \in [0, T]}\mathcal{W}^2_2(\mu^{h}_t, \delta_{0}) )  \nonumber\\
&\leq & C_{L}( 1 + \sup_{t \in [0, T]}|X^{h}_t(x)|^2 + \int_{\mathbb{R}^{d}} \sup_{t \in [0, T]}|X^{h}_t(x)|^2 \mu_{0}(\mathrm{d}x) ) \nonumber\\
&\leq & C_{L, T, N}( 1 + |x|^2 ), \nonumber
\end{eqnarray*}
which implies that $|G(t, X^{h}_t(x), \mu^{h}_t, \theta)| \in L^2([0, T], L^2(\Theta, \vartheta))$ and for any $ 0 \leq s \leq t \leq T$,
\begin{eqnarray*}\label{vI-eq5.14}
       | F_{\epsilon}(t, x) - F_{\epsilon}(s, x)|
&\leq& |\int_s^t \int_{\Theta} G(r, X^{h}_r(x), \mu^{h}_r, \theta )
       \cdot ( h^{\epsilon}(\theta, r) - h(\theta, r))\vartheta(\mathrm{d}\theta)\mathrm{d}r | \nonumber\\
&\leq& (\int_s^t \| |G(r, X^{h}_r(x), \mu^{h}_r, \cdot )| \|_{\vartheta}^2 \mathrm{d}r )^{\frac{1}{2}}
       \cdot (\int_0^T \| h^{\epsilon}(\cdot, r) - h(\cdot, r) \|_{\vartheta}^2 \mathrm{d}r )^{\frac{1}{2}} \nonumber\\
&\leq& C_{L, T, N}( 1 + |x|)(t-s)^{\frac{1}{2}}, \nonumber
\end{eqnarray*}
where we have used the facts that $h^{\epsilon}$, $h \in \mathcal{H}^N$ in the last step. In combination with the fact that $h^{\epsilon}$ converges to $h$ weakly in $L^2([0, T], L^2(\Theta, \vartheta))$, it follows that for every $x \in \mathbb{R}^{d}$,
\begin{equation}\label{vI-eq5.20}
\lim_{\epsilon \rightarrow 0}\sup_{t \in [0, T]}F_{\epsilon}^{2}(t, x) = 0.
\end{equation}
Moreover, by the dominated convergence theorem,
\begin{equation}\label{vI-eq5.21}
\lim_{\epsilon \rightarrow 0}\int_{\mathbb{R}^{d}} \sup_{t \in [0, T]}F_{\epsilon}^{2}(t, x) \mu_{0}(\mathrm{d}x) = 0.
\end{equation}
Therefore, by \eqref{vI-eq5.20} and \eqref{vI-eq5.21}, letting $\epsilon \rightarrow 0$ in \eqref{vI-eq5.12}, we get \eqref{vI-eq5.19}.

Finally, combing with \eqref{vI-eq5.16} and \eqref{vI-eq5.18}, letting $\epsilon \rightarrow 0$, and then letting $M \rightarrow \infty$ in \eqref{vI-eq5.1}, we obtain
\begin{equation*}\label{vI-eq5.22}
\lim_{\epsilon \rightarrow 0}\sup_{t \in [0, T]}\sup_{x \in \mathbb{R}^{d}} \frac{|X^{h^\epsilon}_t(x) - X^h_t(x)|}{1 + |x|^{1+\delta}}=0,
\end{equation*}
completing the proof.

\end{proof}

\end{document}